\documentclass[12pt]{amsart}
\usepackage{amsmath,amssymb,latexsym,cancel,rotating}

\usepackage{tikz-cd}

\usetikzlibrary{cd}

\usepackage{setspace}
\usepackage{dcpic,pictexwd}
\usepackage[all]{xy}
\usepackage[pdftex]{hyperref}
\usepackage{bbm}

\makeatletter
\@namedef{subjclassname@2020}{%
  \textup{2020} Mathematics Subject Classification}
\makeatother

\numberwithin{equation}{section}

\newtheorem{theorem}{Theorem}[section]
\newtheorem{proposition}[theorem]{Proposition}

\newtheorem{lemma}[theorem]{Lemma}

\theoremstyle{definition}

\newtheorem{remark}[theorem]{Remark}

\newtheorem{definition}[theorem]{Definition}

\input xy
\xyoption{poly}
\xyoption{2cell}
\xyoption{all}

\def\Ext{\text{Ext}}
\def\Hom{\text{Hom}}
\def\End{\text{End}}
\def\ker{\text{Ker}}

\def\dim{\text{dim}}
\def\rad{\text{rad}}

\def\soc{\text{Soc}}

\def\ZZ{\mathbb{Z}}

\def\Acal{\mathcal{A}}

\def\Fcal{\mathcal{F}}

\def\dim{\text{dim}}

\def\Hom{\text{Hom}}
\def\Ext{\text{Ext}}

\def\half{\frac{1}{2}}

\renewcommand{\eqref}[1]{{\rm (\ref{#1})}}

\begin{document}

\title[Kronecker quantum cluster algebra with principal coefficients]
{Recursive formulas for the Kronecker quantum cluster algebra with principal coefficients}

\author{Ming Ding, Fan Xu and Xueqing Chen}

\address{School of Mathematics and Information Science\\
Guangzhou University, Guangzhou, P. R. China}
\email{m-ding04@mails.tsinghua.edu.cn (M.Ding)}
\address{Department of Mathematical Sciences\\
Tsinghua University\\
Beijing 100084, P. R. China} \email{fanxu@mail.tsinghua.edu.cn (F.Xu)}
\address{Department of Mathematics,
 University of Wisconsin-Whitewater\\
800 W. Main Street, Whitewater, WI.53190. USA}
\email{chenx@uww.edu (X.Chen)}




\maketitle

\begin{abstract} We use the quantum version of Chebyshev polynomials  to explicitly construct the recursive formulas for the Kronecker quantum cluster algebra  with principal coefficients.  
As a byproduct, we obtain two  bar-invariant positive $\mathbb{ZP}$-bases with one being the atomic basis.
\end{abstract}


\section{Introduction}
Cluster algebras were invented by Fomin and Zelevinsky  \cite{ca1}, and their quantum analogues (called quantum cluster algebras) were introduced  by Berenstein and Zelevinsky  \cite{berzel}.  A  (quantum)  cluster algebra   is generated by distinguished generators called (quantum) cluster variables which are grouped into overlapping subsets called clusters.  The first important result is the Laurent phenomenon which asserts that  any (quantum)  cluster algebra element is a Laurent polynomial in the cluster variables in any cluster. A (quantum)  cluster algebra element is said to be positive if the Laurent polynomial has non-negative coefficients.

In order to obtain more structural results in cluster theory, one needs to  construct cluster multiplication formulas. For the classical cluster algebras of acyclic quivers, Sherman and Zelevinsky \cite{SZ} firstly provided the cluster multiplication formulas in rank 2 cluster algebras of finite and affine types. This result was generalized to rank $3$ cluster algebra of affine type $A_{2}^{(1)}$ by Cerulli \cite{CI-1}.  Caldero and Keller \cite{ck}  constructed the cluster multiplication formulas  between two generalized cluster variables for simply laced Dynkin quivers, which were generalized to affine types in~\cite{Hubery} and to acyclic types in~\cite{XX,X}.  In the quantum setting, Ding and Xu \cite{dx} firstly gave the cluster multiplication formulas of the quantum  cluster algebra of type  $A_{1}^{(1)}$ without coefficients, which were generalized to  type $A_{2}^{(2)}$ by Bai, Chen, Ding and Xu in~\cite{BCDX}.
Recently, Chen, Ding and Zhang \cite{cdz} constructed the cluster multiplication formulas  in the acyclic quantum cluster algebras with arbitrary coefficients through some quotients of derived Hall algebras of acyclic valued quivers. Note that, for general cases, although the  structure constants  deduced from cluster multiplication formulas have the homological interpretations, it is still difficult to explicitly compute them in these formulas.

Cluster multiplication formulas play a very important role in  constructing ``good" bases of cluster algebras.  A basis of a cluster algebra is said to be positive if its structure constants are positive,  and to be atomic (or canonical) if  positive elements are exactly linear combinations of  basis elements with non-negative coefficients. Note that the atomic basis must be a positive basis. 
It attracts a lot of attention to find ``good" bases such as the atomic basis or canonical basis of a cluster algebra. Sherman and Zelevinsky~\cite{SZ} introduced and constructed atomic bases in rank 2 cluster algebras of finite and affine types, which were originally called canonical bases. Cerulli constructed the atomic basis in the rank $3$ cluster algebra of affine type $A_{2}^{(1)}$ in~\cite{CI-1}. He also proved the coincidence of the atomic basis and the set of cluster monomials for cluster algebras of finite type~\cite{CI-2}. For non-finite type cluster algebras, the set of cluster monomials does not form a basis even it is linearly independent. Dupont and Thomas~\cite{DTh} constructed the atomic bases of cluster algebras of affine type $\widetilde A$. They also provided a new proof of Cerulli's result for cluster algebras of type $A$.
In the quantum setting, the atomic bases of the quantum cluster algebras  of  type $A_{1}^{(1)}$ (or $\widetilde{A}_{1,1}$), $A_{2}^{(2)}$ and $\widetilde{A}_{2n-1,1}$  without coefficients were constructed in \cite{dx, BCDX, dxc} by using the explicit cluster multiplication formulas.  However,  most of  quantum cluster algebras must have coefficients involved. It is natural to ask whether there exists an explicit treatment of  the atomic bases for the acyclic quantum cluster algebras with coefficients by using representation theory of quivers.  Sherman and Zelevinsky \cite{SZ} initiated the use of Chebyshev polynomials to construct the atomic bases of rank $2$ cluster algebras of affine type. In order to obtain  atomic bases of cluster algebras with principle coefficients, Dupont~\cite{Dup0}  introduced the deformed Chebyshev polynomials and conjectured that  the deformed Chebyshev polynomials have interactions with atomic bases or canonically positive bases in affine cluster algebras. He proved this conjecture for cluster algebra of affine type $A_{1}^{(1)}$.  Note that for the $A_{2}^{(1)}$ case, Cerulli~\cite{CI-1}  already used the deformed Chebyshev polynomials to prove the existence of the atomic basis.

As the Chebyshev polynomials  are used to construct the positive bases in classical cluster algebras, it is important to study them in quantized cases. In this paper, we modify the definition of Chebyshev polynomials which can be viewed as the quantum analogue of those in \cite{CI-1, Dup0}.  Then we explicitly construct the recursive formulas for the Kronecker quantum cluster algebra with principal coefficients based on the representation theory of the Kronecker quiver.  As a direct corollary, we obtain two  bar-invariant positive $\mathbb{ZP}$-bases with  one being the atomic basis. We expect to apply results and methods used in this paper in the future studies for all affine types.
  
  The paper is organized as follows:  in Section 2, we collect basic concepts of quantum cluster algebras and quantum cluster characters and recall some multiplication formulas; Section 3 provides detailed calculations of recursive formulas in the Kronecker quantum cluster algebra  with principle coefficients; in Section 4, we obtain an explicit  treatment of  two bar-invariant positive $\mathbb{ZP}$-bases in terms of quiver representations.

\section{Quantum cluster algebra and quantum cluster character}
We recall some basic features of quantum cluster algebras and quantum cluster characters. 
\subsection{Quantum cluster algebra} 
 Let $\Lambda: \mathbb{Z}^{m} \times \mathbb{Z}^{m} \rightarrow \mathbb{Z}$ be a skew-symmetric bilinear form on $ \mathbb{Z}^{m}$. The  natural basis  of  $\mathbb{Z}^{m}$ is denoted by $\{{\bf e}_1,\cdots,{\bf e}_m\}$.  Denote by  $\ZZ[\mathfrak{q}^{\pm\frac{1}{2}}]$ be the ring of integer Laurent polynomials of an indeterminate $\mathfrak{q}$. 
 The  quantum torus  $\mathcal{T}_{\mathfrak{q}}$ associated to $\Lambda$ as the $\ZZ[\mathfrak{q}^{\pm\frac{1}{2}}]$-algebra is freely generated by the set $\{X^{\bf e}: {\bf e}\in \mathbb{Z}^{m}\}$ with the following
twisted product:
$$X^{\bf e}\cdot X^{\bf f}=\mathfrak{q}^{\Lambda({\bf e},{\bf f})/2}X^{{\bf e+f}}\,\,\, \,\,\,\, \text{ for any }\bf{e}, \bf{f} \in  \mathbb{Z}^m.$$
In what follows, we omit the twisted product $\cdot$ for the multiplication in algebras. The skew-field of fractions for $\mathcal{T}_{\mathfrak{q}}$ is  denoted by $\mathcal{F}_{\mathfrak{q}}$. It will cause no confusion if we use the same letter $\Lambda$ to designate the $m\times m$ skew-symmetric integer matrix associated to the bilinear form $\Lambda$.

Let $\widetilde{B}=(b_{ij})$ be an $m\times n$ integer matrix with $n\le m$ and $\widetilde{B}^{tr}$ the transpose of  $\widetilde{B}$.  The {\em principle} part of $\widetilde{B}$ is the upper $n\times n$ submatrix of $\widetilde{B}$ denoted by $B$.  We call  the pair
$(\Lambda, \widetilde{B})$ {\em compatible} if $\widetilde{B}^{tr}\Lambda=(D|0)$ for some diagonal matrix
$D$ with positive entries. An {\em initial  quantum seed} $(\Lambda, \widetilde{B}, X)$ of $\mathcal{F}_{\mathfrak{q}}$ 
  consists of a compatible pair $(\Lambda,\widetilde{B})$ and the set  $X=\{X_1,\cdots,X_m\}$ with  $X_i:=X^{{\bf e_i}}$ for $1\leq i\leq m$. For any $1\leq k\leq n$, the mutation $\mu_k$ of the quantum seed $(\Lambda, \widetilde{B}, X)$ in direction $k$  is the new quantum seed $\mu_k (\Lambda, \widetilde{B}, X) :=(\Lambda',\widetilde{B}',X')$ given as follows:
\begin{enumerate}
\item $\Lambda'=E^{tr}\Lambda E$, where the
$m\times m$ matrix $E=(e_{ij})$ is given by
\[e_{ij}=\begin{cases}
\delta_{ij} & \text{if $j\ne k$;}\\
-1 & \text{if $i=j=k$;}\\
[-b_{ik}]_{+} & \text{if $i\ne j = k$.}
\end{cases}
\]
\item \ $\widetilde{B}'=(b'_{ij})$ with
\[b'_{ij}=\begin{cases}
-b_{ij} & \text{if $i=k$ or $j=k$;}\\
b_{ij}+[b_{ik}]_{+}b_{kj} +b_{ik}[-b_{kj}]_{+}& \text{otherwise.}
\end{cases}
\]
\item $X'=\{X'_1,\cdots,X'_m\}$ with
\[ X_i'=\begin{cases}
X_i & \text{ if }    i\neq k; \nonumber\\
X^{\sum_{1\leq j\leq m}[b_{jk}]_{+} {\bf e}_j-{\bf e}_k}+X^{\sum_{1\leq j\leq m}[-b_{jk}]_{+} {\bf e}_j -{\bf e}_k} & \text{ if }    i=k,\nonumber
\end{cases}
\]
\end{enumerate}  
where $[a]_{+}:=\max\{0,a\}$  for $a \in \mathbb{Z}$.  
  
Note that the mutation is an involution, i.e., $\mu_k^2 (\Lambda, \widetilde{B}, X) =(\Lambda,\widetilde{B},X)$. Thus we have an equivalence relation: the quantum seed  $(\Lambda, \widetilde{B}, X)$ is
 {\em mutation-equivalent} to $(\Lambda', \widetilde{B}', X')$, denoted by $(\Lambda, \widetilde{B}, X)\sim(\Lambda', \widetilde{B}', X')$, if they can be obtained from each other
by finite sequence of mutations. In the context of a quantum seed $(\Lambda', \widetilde{B}', X')$, we call the set $\{X'_i~|1\leq i\leq n\}$ the {\em cluster} of the seed and call the elements $X'_i$
the {\em  cluster variables}. The  {\em  cluster monomials} are the elements $X'^{\bf e}$ with $\bf e \in (\mathbb{Z}_{\geq 0})^{m}$  and 
$e_i=0$ for $n+1\leq i\leq m$. The
elements in the set $\mathbb{P}:=\{X_i~|~n+1\leq i\leq m\}$ are called the {\em coefficients}. Let 
$\ZZ\mathbb{P}$ be the ring of  Laurent polynomials in the elements of $\mathbb{P}$ with coefficients in $\ZZ[\mathfrak{q}^{\pm\frac{1}{2}}]$. The
{\em quantum cluster algebra}
$\mathcal{A}_{\mathfrak{q}}(\Lambda,\widetilde{B})$ is defined as the
$\ZZ\mathbb{P}$-subalgebra of $\mathcal{F}_{\mathfrak{q}}$ generated by all
 cluster variables.
On $\mathcal{T}_{\mathfrak{q}}$, we can define the $\mathbb{Z}$-linear bar-involution: 
$$\overline{\mathfrak{q}^{\frac{r}{2}}X^{\mathbf{e}}}=\mathfrak{q}^{-\frac{r}{2}}X^{\mathbf{e}},\ \   \text{ for  any}\   r\in \mathbb{Z}  \text { and }\mathbf{e}\in \mathbb{Z}^{m}.$$
It is straightforward to show that $\overline{XY} = \overline{Y}\ \overline{X}$ for any $X, Y \in \mathcal{T}_{\mathfrak{q}}$ and that 
quantum cluster monomials are bar-invariant. The bar-involution on $\mathcal{A}_{\mathfrak{q}}(\Lambda,\widetilde{B})$ can be naturally induced.

The following celebrated quantum Laurent phenomenon proved by Berenstein and Zelevinsky
is one of the most important structural results on quantum cluster algebras, which asserts that every cluster variable can be written as a Laurent polynomial of cluster variables in any other cluster.
\begin{theorem}\cite[Theorem 5.1]{berzel}\label{laurent}
The quantum cluster algebra $\Acal_{\mathfrak{q}}(\Lambda,\widetilde{B})$ is a
subalgebra of the ring of  Laurent polynomials in the cluster variables in any cluster over $\ZZ\mathbb{P}$.
\end{theorem}

\subsection{Quantum cluster character}
Let $k=\mathbb{F}_q$ be a finite field of order $q$. Fix two positive integers $n \leq m$. Let  $\widetilde{Q}$ be an acyclic valued quiver with vertex set $\{1,2, \dots, m\}$. The full subquiver $Q$ on the vertices $\{1,\ldots,n\}$ is called the principal part of $\widetilde{Q}$.  The elements in the subset $\{n+1,\dots,m\}$ are called the frozen vertices.

Let $\mathfrak{S} =k Q$  and $\widetilde{\mathfrak{S}}=k\widetilde{Q}$ be the  path algebras of $Q$ and $\widetilde{Q}$, respectively. Let
$S_i$, $P_i$ and $I_i$ be the simple, projective and injective $\widetilde{\mathfrak{S}}$-module associated to the vertex $i$ for any  $1\leq i\leq m$, respectively.
The $m\times n$ matrix $\widetilde{B}=(b_{ij})$ associated to $\widetilde{Q}$ is defined by
$$
b_{ij}=\dim_{\End_{\widetilde{\mathfrak{S}}}(S_{i})^{op}}\Ext^{1}_{\widetilde{\mathfrak{S}}}(S_{i},S_{j})-
\dim_{\End_{\widetilde{\mathfrak{S}}}(S_{i})}\Ext^{1}_{\widetilde{\mathfrak{S}}}(S_{j},S_{i})
$$
for $1\leq i\leq m$, $1\leq j\leq n$.  Note that the upper $n \times n$ part of $\widetilde{B}$  is $B$. The
left $m\times n$ submatrix of the identity matrix of size $m\times
m$ is denoted by $\widetilde{I}$.

Assume that there exists some skew-symmetric
 integer matrix $\Lambda$ of size $m\times m$ such that
\begin{align*}
\Lambda(-\widetilde{B})=\left(\begin{array}{c} D_n\\
0 \end{array}\right),\nonumber
\end{align*}
where $D_n=\operatorname{diag}(d_1,\cdots,d_n)$ with $d_i\in \mathbb{N}$ for
$1\leq i\leq n$.  Note that if we  enlarge the quiver $Q$ to $\widetilde{Q}$, where we attach principal frozen vertices to the valued quiver $Q$,  such a matrix $\Lambda$ always exists  \cite{z}.

The quantum cluster algebra associated to this obtained compatible pair $(\Lambda,\widetilde{B})$  is simply denoted by
$\mathcal{A}_{\mathfrak{q}}(Q)$.

Let $\widetilde{R}=\widetilde{R}_{\widetilde{Q}}=(\widetilde{r}_{ij}) $
be the $m\times n$ matrix with 
\[
\widetilde{r}_{ij}:=\mathrm{dim}_{\End_{\widetilde{\mathfrak{S}}}(S_{i})}\Ext^{1}_{\widetilde{\mathfrak{S}}}(S_j,S_i)
\] for $1\leq i\leq m$ and $1\leq j\leq n$. And define
$\widetilde{R}^{'}:=\widetilde{R}_{\widetilde{Q}^{op}}.$
Note that the principal part of  $\widetilde{R}$ is $R$, 
$\widetilde{B}=\widetilde{R}^{'}-\widetilde{R}$ and $B=R^{'}-R$.

Let $\mathcal C_{\widetilde{Q}}$ be the cluster category of
the valued quiver $\widetilde{Q}$ with the shift functor 
$[1]$ and the Aulander-Reiten translation functor  $\tau$. For more details about cluster category,  the readers can refer to~\cite{BMRRT, rupel-2}. Note that each object $M$ of $\mathcal
C_{\widetilde{Q}}$ can be uniquely decomposed up to isomorphism as follows:
$$M=M_0\oplus P_M[1]$$
where $M_0$ is an $\widetilde{\mathfrak{S}}$-module and $P_M$ is a projective
$\widetilde{\mathfrak{S}}$-module. Let $P_M=\bigoplus_{1\leq i \leq m}m_iP_i$.  The definition of dimension vector $\mathrm{\underline{dim}}$ on
$\widetilde{\mathfrak{S}}$-modules  can be extended on objects of $\mathcal C_{\widetilde{Q}}$
by defining 
$$\mathrm{\underline{dim}}M=\mathrm{\underline{dim}}M_0-(m_i)_{1\leq i \leq m}.$$

In the following, we will denote the dimension vector of an
$\mathfrak{S}$-module $X$ by the corresponding underlined
lower case letter $\underline{x}$ viewed as a column
vector in $\mathbb{Z}^n.$ For simplifying notations, we write $\mathrm{Ext}^{1}(M,N):=\mathrm{Ext}_{\widetilde{\mathfrak{S}}}^{1}(M,N)$ and
$\mathrm{Hom}(M,N):=\mathrm{Hom}_{\widetilde{\mathfrak{S}}}(M,N)$.
The Euler form on $\widetilde{\mathfrak{S}}$-modules $M$ and $N$ is given by
$$\langle M,N\rangle=\mathrm{dim}_{k}\mathrm{Hom}(M,N)-\mathrm{dim}_{k}\mathrm{Ext}^{1}(M,N).$$

Let $\mathcal{A}_{q}(Q)$ be the specialization of the quantum cluster algebra $\mathcal{A}_{\mathfrak{q}}(Q)$ by evaluating $\mathfrak{q}=q$. 
For an $\mathfrak{S}$-module $N$, we denote its socle and radical by $\operatorname{Soc}\,N$ and $\operatorname{rad}\,N$.
Let $M$ be an $\mathfrak{S}$-module and $I$ an injective
$\widetilde{\mathfrak{S}}$-module, then the quantum cluster 
character of $\mathcal{A}_{q}(Q)$ is defined as:
                    $$
                       X_{M\oplus I[-1]}=\sum_{\underline{e}} |\mathrm{Gr}_{\underline{e}} M|q^{-\frac{1}{2}
\langle
\underline{e},\underline{m}-\underline{e}-\underline{i}\rangle}X^{-\widetilde{B}\underline{e}-(\widetilde{I}-\widetilde{R}^{'})\underline{m}+\underline{\mathrm{dim}}\, 
\soc \, I},
                    $$
where $\underline{\operatorname{dim}}\, I= \underline{i},
\underline{\mathrm{dim}}\,M= \underline{m}$ and
$\mathrm{Gr}_{\underline{e}}M$ is the Grassmannian of sub-representations of $M$ with dimension
vector  $\underline{e}$. Note
that  for any projective $\widetilde{\mathfrak{S}}$-module $P$ and
injective $\widetilde{\mathfrak{S}}$-module $I$ with $\soc \, I=P/\rad
\, P$, 
we have that
$$
X_{P[1]}=X_{\tau P}=X^{\underline{\mathrm{dim}} \, P/\rad
\, P}=X^{\underline{\mathrm{dim}}\, \soc \, I}=X_{I[-1]}=X_{\tau^{-1}I}.
$$

The main
results in \cite{fanqin, rupel-2} show that $\mathcal{A}_{{q}}(Q)$ as the
$\mathbb{ZP}$-subalgebra of $\Fcal_{q}$ is generated by cluster variables
$$\{X_M| M\ \text{is an  indecomposable rigid $\mathfrak{S}$-module}\}\sqcup 
\{X_{I_i[-1]} | 1\leq i\leq n\}.$$

Let $M$ and $N$ be two $\mathfrak{S}$-modules and $\theta:N\to \tau M$ be a morphism, we can deduce an exact sequence
\[
  0 \longrightarrow D \longrightarrow  N  \stackrel{\theta}{\longrightarrow} \tau M   \longrightarrow \tau A \oplus I \longrightarrow 0  
 \]
 where $D=\operatorname{ker} \theta$, $\tau A \oplus I=\operatorname{coker} \theta$, $I$ is an injective $\widetilde{\mathfrak{S}}$-module, $A$ and $M$ have the same maximal projective summand.

The following theorems are the generalizations of those in \cite{dx-2, fanqin} to acyclic valued quivers.
\begin{theorem}\cite[Theorem 4.5]{rupel-2}\label{m-1}
Assume  $M$ and $N$ are $\mathfrak{S}$-modules with a unique (up to scalar) nontrivial extension $E\in\Ext^1 (M,N)$, in particular $\dim_{\End(M)}\Ext^1 (M,N)$$=1$.  Let $\theta\in\Hom(N,\tau M)$ be the equivalent morphism with $A, D, I$ described as above.  Furthermore assume that $\Hom (A\oplus D,I)=0=\Ext^1 (A,D)$.  Then we have that
 \begin{align*}
 & X_MX_N \\
   = &q^{\half\Lambda((\widetilde{I}-\widetilde{R}^{'}){\underline{m}},(\widetilde{I}-\widetilde{R}^{'}){\underline{n}})}X_E + q^{\half\Lambda((\widetilde{I}-\widetilde{R}^{'}){\underline{m}},(\widetilde{I}-\widetilde{R}^{'}){\underline{n}})+\half\langle{M},{N}\rangle-\half\langle {A}, {D}\rangle}X_{D\oplus A\oplus I[-1]}.
\end{align*} 
 \end{theorem}
 
 Let $M$ be an $\mathfrak{S}$-module and $I$ an injective $\widetilde{\mathfrak{S}}$-module.  Let $\nu$ be the Nakayama functor on $\widetilde{\mathfrak{S}}$-modules and  
 write $P=\nu^{-1}(I)$.   From morphisms $f: M\to I$ and $g:P\to M$, we can deduce two exact sequences
 \[ 0 \longrightarrow G \longrightarrow  M  \stackrel{f}{\longrightarrow} I   \longrightarrow I' \longrightarrow 0  \]
 and
  \[ 0 \longrightarrow P' \longrightarrow P  \stackrel{g}{\longrightarrow} M  \longrightarrow F \longrightarrow 0  \]
   where $G=\ker f$, $I'=\operatorname{coker} f$ is an injective $\widetilde{\mathfrak{S}}$-module, $P'=\ker g$ is a projective $\widetilde{\mathfrak{S}}$-module and $F=\operatorname{coker} g$.

\begin{theorem}\cite[Theorem 4.8]{rupel-2}\label{m-2}
Let $M$ be an $\mathfrak{S}$-module, $I$ and $P$  be $\widetilde{\mathfrak{S}}$-modules defined as above.  Assume that there exist unique (up to scalar) morphisms $f\in\Hom(M,I)$ and $g\in\Hom(P,M)$, in particular $$\dim_{\End(I)}\Hom(M,I)=\dim_{\End(P)}\Hom(P,M)=1.$$  Define $F$, $G$, $I'$, $P'$ as above and assume further that $\Hom(P',F)=\Hom(G,I')=0$.  Then we have
\begin{align*}
& X_MX_{I[-1]}\\
= & q^{-\half\Lambda((\widetilde{I}-\widetilde{R}^{'}){\underline{m}},\underline{\mathrm{dim}}\, 
\soc \, I)}X_{G\oplus I'[-1]}+ q^{-\half\Lambda((\widetilde{I}-\widetilde{R}^{'}){\underline{m}}, \underline{\mathrm{dim}}\, 
\soc \, I)-\half\mathrm{dim}_{k}\operatorname{End}(I)}X_{F\oplus P'[1]}.
\end{align*}   
\end{theorem}

\section{Recursive formulas for the Kronecker quantum cluster algebra  with principle coefficients}

Consider the Kronecker quiver $Q$, i.e., the quiver  of type
$\widetilde{A}_{1,1}$:
$$
\xymatrix {1\bullet  \ar @<2pt>[r] \ar @<-2pt>[r]& \bullet 2}
$$

It is well-known \cite{DR} that indecomposable $kQ$-modules
are grouped into (up to isomorphism) three families: the indecomposable preprojective
modules with dimension vector $(l,l+1)$ for $l\in \mathbb{Z}_{\geq 0}$ denoted by $V(-l)$,
the indecomposable regular modules (in
particular, denote by $R(l)$ the indecomposable regular module
with dimension vector $(l, l)$ for $l \in \mathbb{Z}_{\geq 1}$ from the same homogeneous tube whose mouth has dimension vector $(1,1)$), and the indecomposable preinjective modules
with dimension vector $(l+1,l)$ for $l\in \mathbb{Z}_{\geq 0}$ denoted by $V(l+3)$. 

The preprojective  component looks like the following:

$$
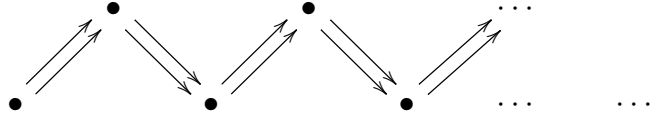
{\xymatrix{
& \bullet  \ar@<-0.5ex>[dr]  \ar@<0.5ex>[dr]  &   & \bullet  \ar@<-0.5ex>[dr]  \ar@<0.5ex>[dr]      & & \cdots & \\
\bullet  \ar@<-0.5ex>[ur] \ar@<0.5ex>[ur] &                                                                   & \bullet \ar@<-0.5ex>[ur] \ar@<0.5ex>[ur] &            & \bullet   \ar@<-0.5ex>[ur] \ar@<0.5ex>[ur] & \cdots  & \cdots }
}$$

The preinjective components looks like the following:

$$
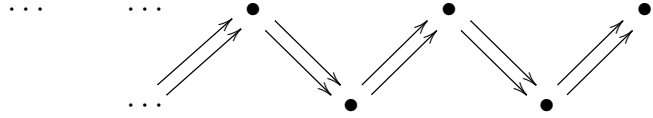
{\xymatrix{
\cdots  & \cdots &\bullet \ar@<-0.5ex>[dr]  \ar@<0.5ex>[dr]  &   & \bullet  \ar@<-0.5ex>[dr]  \ar@<0.5ex>[dr]      & & \bullet \\
& \cdots \ar@<-0.5ex>[ur] \ar@<0.5ex>[ur] &                                                                   & \bullet \ar@<-0.5ex>[ur] \ar@<0.5ex>[ur] &            & \bullet   \ar@<-0.5ex>[ur] \ar@<0.5ex>[ur] &  }
}$$

We consider the following ice
quiver $\widetilde{Q}$ with frozen vertices $3$ and $4$:
$$
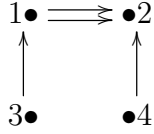
\xymatrix{1\bullet \ar @<2pt>[r] \ar @<-2pt>[r]  & \bullet 2\\
3 \bullet \ar[u]&   \bullet 4\ar[u]}$$ Thus we have
$$\widetilde{R}^{'}=\left(\begin{array}{cc} 0 & 2\\
0& 0\\
1& 0\\
0 & 1\end{array}\right),\ 
\widetilde{R}=\left(\begin{array}{cc} 0 & 0\\
2& 0\\
0& 0\\
0 & 0\end{array}\right),\ 
\widetilde{I}-\widetilde{R}^{'}=\left(\begin{array}{cc} 1 & -2\\
0& 1\\
-1& 0\\
0 & -1\end{array}\right),\ 
\widetilde{B}=\left(\begin{array}{cc} 0 & 2\\
-2& 0\\
1& 0\\
0 & 1\end{array}\right).$$

An easy calculation shows that the following skew-symmetric $4\times
4$ integer
matrix $$\Lambda=\left(\begin{array}{cccc} 0 & 0& -1& 0\\
0 & 0& 0& -1\\
1 & 0& 0& -2\\
0 & 1& 2& 0\end{array}\right)$$ satisfying
\begin{align}\label{eq:simply_laced_compatible}
\Lambda(-\widetilde{B})=\left(\begin{array}{c} D_2\\
0 \end{array}\right),\nonumber
\end{align}
where $D_2=\left(\begin{array}{cc} 1 & 0\\
0 & 1\end{array}\right)$. The quantum cluster algebra $\mathcal{A}_{q}(Q)$ associated to this pair $(\Lambda,\widetilde{B})$ is called the Kronecker quantum cluster algebra  with principle coefficients. We denote the coefficients by 
$y_1=X_{I_3[-1]},\ y_2=X_{I_4[-1]},$
and all cluster variables by 
$x_{l}=X_{V(l)}$ for $l\in \ZZ$
where $x_{1}=X_{V(1)}=X_{I_1[-1]}$ and $x_{2}=X_{V(2)}=X_{I_2[-1]}$ are initial cluster variables.

By the definition of quantum cluster characters, it follows  that $$X_{R(1)}=X^{(1,-1,1,1)}+X^{(-1,-1,1,0)}+X^{(-1,1,0,0)}.$$
Therefore the expression of $X_{R(1)}$ is independent of the choice of the indecomposable regular module with dimension vector $(1,1)$ and then we denote $X_{R(1)}$ by  $X_\delta$.  Note that the element $X_\delta$ coincides with 
the quantum loop element $s_1$ defined in \cite{cla} in terms of perfect matchings associated to  the Kronecker quantum cluster algebra  with principle coefficients. A direct calculation shows that 
$$X_\delta= x_0 x_3 -q^{\frac{3}{2}} x_1 x_2 y_2,$$ and thus  $X_\delta$ belongs to $\Acal_{q}(Q)$ (cf. \cite{cla}).

The $n$-th Chebyshev polynomials of the second kind $S_{n}(x)$ appeared in \cite{cz} is related to  the dual semi-canonical basis of the Kronecker cluster algebra. We modify the  definition  of $S_{n}(x)$ for $n\in \mathbb{Z}_{\geq 0}$ on $X_\delta$ as follows
  $$
  S_0(X_\delta)=1,S_1(X_\delta)=X_\delta, S_2(X_\delta)=S_1(X_\delta)S_1(X_\delta)-X^{(0,0,1,1)},$$
  $$S_{n+1}(X_\delta)=S_n(X_\delta)S_1(X_\delta)-X^{(0,0,1,1)} S_{n-1}(X_\delta), \ \ for\ n\geq 2.
 $$
In the following, we denote $S_n(X_\delta)$ by $S_n$ for simplicity. According to  the definition of $S_n$ and the fact that $S_1\in \mathcal{A}_q (Q)$, it is easy to see  $S_n\in \mathcal{A}_q (Q)$ for any $n\geq 1$. 

We start with some technical lemmas:
\begin{lemma}\label{coef}  For any  $n\geq 0$ and  $m\geq 0$, we have that
\begin{itemize}
\item[(1)] $y_1 S_n=q^{-n} S_n y_1$;
\item[(2)] $y_2 S_n=q^n S_ny_2$;
\item[(3)] $X^{(0,0,1,1)} S_n=S_n X^{(0,0,1,1)}$;
\item[(4)] $S_mS_n=S_n S_m.$
\end{itemize} 
\end{lemma}
\begin{proof}
(1)\  When $n=0$, it is trivial. When $n=1$, we have 
\begin{align*}
 y_1 S_1 = &  X^{(0,0,1,0)} (X^{(1,-1,1,1)} +X^{(-1,-1,1,0)} + X^{(-1,1,0,0)})    \\
= &  q^{-1}  (X^{(1,-1,1,1)} +X^{(-1,-1,1,0)} + X^{(-1,1,0,0)})  X^{(0,0,1,0)} \\
=& q^{-1} S_1  y_1.
\end{align*} 
By induction, we have 
\begin{align*}
&y_1 S_{n+1}=y_1( S_n S_1 - X^{(0,0,1,1)} S_{n-1}) \\
=& q^{-(n+1)} S_n S_1 y_1 - q^{-2} X^{(0,0,1,1)} y_1 S_{n-1}\\
=& q^{-(n+1)} S_n S_1 y_1-q^{-2} q^{-(n-1)} X^{(0,0,1,1)} S_{n-1}y_1\\
=& q^{-(n+1)} S_{n+1} y_1.
\end{align*} 

(2)\ The proof is similar to (1).

(3)\ Note that $X^{(0,0,1,1)}=q y_1 y_2,$ the proof follows from (1) and (2).

(4)\ According to  (3) and by induction, the statement follows immediately.
\end{proof}

\begin{lemma}\label{bar-s}
For $n\geq 0$, we have that 
\begin{itemize}
\item[(1)] $\overline{S_n} =S_n$;
\item[(2)] $\overline{X^{(0,0,1,1)} S_{n} }= X^{(0,0,1,1)} S_{n}.$
\end{itemize} 
\end{lemma}
\begin{proof}
(1)\ When $n=0$, it is trivial. When $n=1$, by the definition of $X_\delta$, we  see that $\overline{X_\delta}=X_\delta$, i.e, $\overline{S_1}=S_1$.
Then by induction  and Lemma~\ref{coef}, we have
\begin{align*}
 \overline{S_{n+1}}=&   \overline{S_{n}  S_1 - X^{(0,0,1,1)} S_{n-1}}\\ 
 = &  \overline{S_1}\,  \overline{ S_n}- \overline{S_{n-1}} \,  \overline{X^{(0,0,1,1)}} \\
 = &  {S_1} { S_n}- {S_{n-1}} {X^{(0,0,1,1)}} \\
  = &  { S_n}{S_1} -  {X^{(0,0,1,1)}}  {S_{n-1}}\\ 
 =& S_{n+1} 
\end{align*} 

(2)\ By (1) and Lemma ~\ref{coef}, we have $$\overline{X^{(0,0,1,1)} S_{n} }= \overline{S_{n}} \overline{X^{(0,0,1,1)} } =S_{n} X^{(0,0,1,1)}  = X^{(0,0,1,1)} S_{n}.$$
Hence, the proof is finished.
\end{proof}

The following proposition gives multiplication formulas for the modified Chebyshev polynomials of the second kind on 
$X_\delta$.
\begin{proposition}\label{calc1} For $n\geq m\geq 1$, we have that $$S_mS_n=S_{n+m} + X^{(0,0,1,1)} S_{n+m-2} +\cdots   +X^{(0,0,m-1,m-1)} S_{n-m+2} + X^{(0,0,m,m)} S_{n-m}.$$
\end{proposition}
\begin{proof}
When $m=1$, the statement is just the definition. Assume that we have the equality for all positive integers $m$ less than or equal to $k$. Then by Lemma~\ref{coef}, we have
\begin{align*}
 & S_{k+1}S_n =  ({S_1}{ S_k}- {X^{(0,0,1,1)}}S_{k-1} )S_n \\
 = &   S_1(S_{n+k} + X^{(0,0,1,1)} S_{n+k-2} +\cdots   X^{(0,0,k-1,k-1)} S_{n-k+2} + X^{(0,0,k,k)} S_{n-k}) - \\
 & {X^{(0,0,1,1)}}  (S_{n+k-1} + X^{(0,0,1,1)} S_{n+k-3} +\cdots   X^{(0,0,k-2,k-2)} S_{n-k+3} + X^{(0,0,k-1,k-1)} S_{n-k+1}) \\
 =& (S_{n+k+1} + X^{(0,0,1,1)} S_{n+k-1} ) +  X^{(0,0,1,1)}(S_{n+k-1} + X^{(0,0,1,1)} S_{n+k-3} ) + \cdots +\\
 & X^{(0,0,k-1,k-1)}(S_{n-k+3} + X^{(0,0,1,1)} S_{n-k+1} ) +X^{(0,0,k,k)}(S_{n-k+1} + X^{(0,0,1,1)} S_{n-k-1} ) \\
 & -X^{(0,0,1,1)} (S_{n+k-1} + X^{(0,0,1,1)} S_{n+k-3} +\cdots \\
 &   +X^{(0,0,k-2,k-2)} S_{n-k+3} + X^{(0,0,k-1,k-1)} S_{n-k+1})\\
 =&S_{n+k+1} + X^{(0,0,1,1)} S_{n+k-1} +\cdots   +X^{(0,0,k,k)} S_{n-k+1} + X^{(0,0,k+1,k+1)} S_{n-k-1}.     
\end{align*} 
Hence, the proof is finished.
\end{proof}

Now, we  give a representation-theoretic interpretation of the element $S_n$ for any $n\geq 1$.
\begin{proposition}\label{rep}
For $n \geq 1$, we have that
$S_n=X_{R(n)}.$
\end{proposition}
\begin{proof}
The case when $n=1$ is trivial. Note that we have 
\[ 0 \longrightarrow R(1)  \longrightarrow  R(n+1) \longrightarrow  R(n)  \longrightarrow 0 \]
and since  
$\operatorname{dim} \operatorname{Ext}^1 (  R(n) , R(1)) = \operatorname{dim} \operatorname{Hom} (R(1), \tau R(n) )=1$, we have that
\[ 0 \longrightarrow  R(1)   \longrightarrow  \tau R(n)   \longrightarrow    \tau R(n-1)  \oplus I_3 \oplus I_4  \longrightarrow 0. \]
By Theorem \ref{m-1}, we then have that
$$X_{R(n)} X_{R(1)}=X_{R(n+1)} + X_{R(n-1) \oplus (I_3 \oplus I_4)[-1] }.$$
A direct calculation shows that
$$X^{(0,0,1,1)} X^{-\widetilde{B}\underline{e}-(\widetilde{I}- {\widetilde{R}^{'}})(n-1, n-1)^{tr}} = X^{-\widetilde{B}\underline{e}-(\widetilde{I}- {\widetilde{R}^{'}})(n-1, n-1)^{tr} +(0,0,1,1) }, $$
which implies $X_{R(n)} X_{R(1)}=X_{R(n+1)} + X^{(0,0,1,1)} X_{R(n-1)}.$

Note that $S_n S_1 =S_{n+1} + X^{(0,0,1,1)} S_{n-1}$. Therefore, by induction, we can obtain that 
$S_{n+1}=X_{R(n+1)}$.
\end{proof}

\begin{lemma}\label{induc1} In $\mathcal{A}_q (Q)$,  we have that 

\begin{enumerate}
\item $S_1 x_0=x_{-1}+ q^{-\frac{1}{2}} X^{(1,0,0,1)}$;
\item $S_1 x_{-m}=x_{-(m+1)}+ X^{(0,0,1,1)} x_{-(m-1)}$ for $m\geq 1$.
\end{enumerate}
\end{lemma}
\begin{proof}
(1)\ Note that we have  
\[ 0 \longrightarrow V(0) \longrightarrow V(-1) \longrightarrow R(1) \longrightarrow 0 \]
and since $\operatorname{dim} \operatorname{Ext}^1 (  R(1) , V(0)) = \operatorname{dim} \operatorname{Hom} (V(0), \tau R(1) )=1$, we have
\[ 0 \longrightarrow V(0) \longrightarrow \tau R(1)  \longrightarrow I_1 \oplus I_4  \longrightarrow 0. \]
By Theorem \ref{m-1},  we can then obtain that $S_1 x_0= X_{R(1)} X_{V(0)}=x_{-1}+ q^{-\frac{1}{2}} X^{(1,0,0,1)}$;

(2)\  Note that $\operatorname{dim} \operatorname{Ext}^1 (R(1), V(-m))
=\operatorname{dim} \operatorname{Hom} ( V(-m), \tau R(1))=1$,
we have that
\[ 0 \longrightarrow   V(-m) \longrightarrow  V(-m-1) \longrightarrow R(1)  \longrightarrow 0 \]
and 
\[ 0 \longrightarrow V(-m+1) \longrightarrow  V(-m)  \longrightarrow   \tau R(1)   \longrightarrow  I_3 \oplus I_4  \longrightarrow 0. \]
By Theorem \ref{m-1},  we have that
$S_1 x_{-m}=x_{-(m+1)}+ q^{-\frac{1}{2}} X_{{V(-m+1)} \oplus  (I_3 \oplus I_4)[-1] }.$
Note that
$$X^{(0,0,1,1)} X^{-\widetilde{B}\underline{e}-(\widetilde{I}- {\widetilde{R}^{'}})(m-1, m)^{tr}} = q^{-\frac{1}{2}}X^{-\widetilde{B}\underline{e}-(\widetilde{I}- {\widetilde{R}^{'}})(m-1, m)^{tr} +(0,0,1,1) }, $$
which implies $S_1 x_{-m}=x_{-(m+1)}+ X^{(0,0,1,1)} x_{-(m-1)}$ for $m\geq 1$.
\end{proof}

We are now able to prove the following recursive formulas between modified Chebyshev polynomials of the second kind on 
$X_{\delta}$  and cluster variables.
\begin{theorem}\label{calc2}
In $\mathcal{A}_q (Q)$,  we have that 
  \begin{enumerate}
\item $S_n x_0=x_{-n}+ q^{-\frac{1}{2}} S_{n-1} X^{(1,0,0,1)}$, for  $n \geq 1$;
\item $S_n x_{-m}=x_{-(m+n)}+ X^{(0,0,1,1)} S_{n-1}x_{-(m-1)}$, for $n\geq 1$ and $m \geq 1$;
\item $x_1 S_n =x_{n+1}+ q^{-\frac{1}{2}} x_0 y_1S_{n-1}$, for  $n \geq 1$;
\item $x_m S_n=x_{n+m}+  x_{m-1} X^{(0,0,1,1)}S_{n-1} $, for $n\geq 1$ and $m \geq 2$.
\end{enumerate}
\end{theorem}
\begin{proof}
We only prove (1) and (2). The proofs of (3) and (4) are similar.

(1)\ When $n=1$, by Lemma~\ref{induc1} (1), the statement is true. Assume that we have the equality for all positive integers less than or equal to $k$. Then
\begin{align*}
 & S_{k+1}x_0 =  ({S_1}{ S_k}- {X^{(0,0,1,1)}}S_{k-1} )x_0 \\
 = &   S_1 (x_{-k}+ q^{-\frac{1}{2}} S_{k-1} X^{(1,0,0,1)} )       - X^{(0,0,1,1)} ( x_{-(k-1)}+ q^{-\frac{1}{2}} S_{k-2} X^{(1,0,0,1)} )   \\
 =& x_{-(k-1)}+ X^{(0,0,1,1)}x_{-(k-1)} +q^{-\frac{1}{2}} (S_k+ {X^{(0,0,1,1)}}S_{k-2})X^{(1,0,0,1)} \\
 & -X^{(0,0,1,1)} (x_{-(k-1)}+q^{-\frac{1}{2}}  S_{k-2}X^{(1,0,0,1)})\\
 =& x_{-(k+1)}+q^{-\frac{1}{2}}  S_{k}X^{(1,0,0,1)};
  \end{align*} 
(2)\  When $n=1$, by Lemma~\ref{induc1} (2), the statement is true. Assume that we have the equality for all positive integers less than or equal to $k$. Then\begin{align*}
 & S_{k+1}x_{-m} =  ({S_1}{ S_k}- {X^{(0,0,1,1)}}S_{k-1} )x_{-m} \\
 = &   S_1 (x_{-(m+k)}+    {X^{(0,0,1,1)}}S_{k-1} x_{-(m-1)} )    - {X^{(0,0,1,1)}}  (x_{-(m+k-1)}+    {X^{(0,0,1,1)}}S_{k-2} x_{-(m-1)}  )  \\
  =& x_{-(m+k+1)}+ X^{(0,0,1,1)}x_{-(m+k-1)} +{X^{(0,0,1,1)}}(S_k+ {X^{(0,0,1,1)}}S_{k-2} ) x_{-(m-1)}\\
 & -X^{(0,0,1,1)} (x_{-(m+k-1)}+   {X^{(0,0,1,1)}} S_{k-2}x_{-(m-1)}) \\
 =& x_{-(m+k+1)}+ X^{(0,0,1,1)} S_{k} x_{-(m-1)}.
  \end{align*} 
Hence, the proof is finished.
\end{proof}

\begin{remark}
(1)\ In Theorem \ref{calc2} (4), let $n=1$ and $m\geq 2$, then we obtain   $$x_m S_1=x_{m+1}+  x_{m-1} X^{(0,0,1,1)},$$ 
which is exactly the equality in
\cite[Lemma 4.10]{cla}.

(2)\ In Theorem \ref{calc2} (4), let $m=2$ and $n\geq 1$, then we obtain   $$x_2 S_n=x_{n+2}+  x_{1} X^{(0,0,1,1)}S_{n-1}.$$ 
Compare  this equality with 
\cite[Lemma 4.16 (b)]{cla}, it follows that $S_n=s_n$. Thus, by Proposition \ref{rep}, the elements $X_{R(n)}$ give a 
representation-theoretic interpretation of those elements $s_n$ which are expressed in terms of perfect matchings in
\cite{cla}.
\end{remark}

We now prove the following recursive formulas between cluster variables:
\begin{theorem}\label{calc4}  In $\mathcal{A}_q (Q)$,  we have that   
\begin{enumerate}
\item $x_{-(m+n+2)} x_{-m} = q^{-\frac{1}{2}}  S_n X^{(0,0,m,m+1)} + x_{-(n+m+1)}x_{-(m+1)}$, for $m\geq 0$ and $n \geq 0$;
\item $x_{m+n+2} x_{m} = q^{\frac{1}{2}}  S_n X^{(0,0,m,m-1)} + x_{n+m+1}x_{m+1}$, for $m\geq 1$ and $n \geq 0$;
\item $x_0x_{n+2} =q^{\frac{2n+1}{2}} x_1x_{n+1} y_2 + S_n$, for $n\geq 0$;
\item $x_{-m} x_1 =S_{m-1} +  q^{\frac{1}{2}} x_{-(m-1)}  x_0 y_1$, for $m\geq 1$;
\item $x_{-m} x_n =q x_{-(m-1)} x_{n-1}  X^{(0,0,1,1)} + S_{m+n-2}$, for $m\geq 1$ and $n\geq 2$.
\end{enumerate}
\end{theorem}
\begin{proof}
(1)\ Note that
$\operatorname{dim} \operatorname{Ext}^1 (V(-m-2), V(-m))
=\operatorname{dim} \operatorname{Hom} ( V(-m), \tau V(-m-2))=1$.
We have 
\[ 0 \longrightarrow   V(-m) \longrightarrow  2V(-m-1) \longrightarrow V(-m-2)  \longrightarrow 0 \]
and 
\[ 0 \longrightarrow V(-m) \longrightarrow   \tau V(-m-2)  \longrightarrow  mI_3 \oplus (m+1)I_4  \longrightarrow 0. \]
By Theorem \ref{m-1},  we have that
$X_{V(-m-2)} X_{V(-m)} = X^2_{V(-m-1)}+q^{-\frac{1}{2}}  X_{mI_3[-1]\oplus (m+1)I_4[-1]},$
which is exactly the equality
$x_{-(m+2)} x_{-m} = q^{-\frac{1}{2}}  X^{(0,0,m,m+1)} + x^2_{-(m+1)}$. Therefore the statement is true for the case when $n=0$.

When $n=1$, we have 
\begin{align*}
& x_{-(m+3)} x_{-m} = (S_1 x_{-(m+2)} - X^{(0,0,1,1)} x_{-(m+1)} ) x_{-m} \\
=& S_1 (x^2_{-(m+1)} + q^{-\frac{1}{2}}  X^{(0,0,m,m+1)}  ) -X^{(0,0,1,1)} x_{-(m+1)} x_{-m}    \\
=& (x_{-(m+2)}  +   X^{(0,0,1,1)} x_{-m} ) x_{-(m+1)} +q^{-\frac{1}{2}} S_1  X^{(0,0,m,m+1)} - X^{(0,0,1,1)} x_{-(m+1)}  x_{-m}  \\
 =&  x_{-(m+2)}  x_{-(m+1)}  + q^{-\frac{1}{2}} S_1 X^{(0,0,m,m+1)}.
      \end{align*} 
By induction, we then have
\begin{align*}
& q^{-\frac{1}{2}}  S_{n+1} X^{(0,0,m,m+1)} =q^{-\frac{1}{2}}   S_1S_n  X^{(0,0,m,m+1)}  -q^{-\frac{1}{2}} X^{(0,0,1,1)} S_{n-1}  X^{(0,0,m,m+1)}   \\
=& S_1 ( x_{-(m+n+2)}  x_{-m} -  x_{-(m+n+1)}  x_{-(m+1)} ) \\
&- X^{(0,0,1,1)} ( x_{-(m+n+1)}  x_{-m} -  x_{-(m+n)}  x_{-(m+1)} ) \\
=& (x_{-(m+n+3)}   + X^{(0,0,1,1)} x_{-(m+n+1)} ) x_{-m}  -    (x_{-(m+n+2)}   + X^{(0,0,1,1)} x_{-(m+n)} ) x_{-(m+1)}    \\
&-  X^{(0,0,1,1)} (x_{-(m+n+1)}  x_{-m} -x_{-(m+n)}  x_{-(m+1)})  \\
=& x_{-(m+n+3)}  x_{-m} -  x_{-(m+n+2)}  x_{-(m+1)}.
\end{align*} 

(2)\ When $n=0$, it follows from Theorems~\ref{m-1} and~\ref{m-2}. When $n=1$, we have 
\begin{align*}
& x_{m+3} x_{m} = (S_1 x_{m+2} - X^{(0,0,1,1)} x_{m+1} ) x_{m} \\
=& S_1 (x^2_{m+1} + q^{\frac{1}{2}}  X^{(0,0,m,m-1)}  ) -X^{(0,0,1,1)} x_{m+1} x_{m}    \\
=& (x_{m+2}  +   X^{(0,0,1,1)} x_{m} ) x_{m+1} +q^{\frac{1}{2}} S_1  X^{(0,0,m,m-1)} - X^{(0,0,1,1)} x_{m+1}  x_{m}  \\
 =&  x_{m+2}  x_{m+1}  + q^{\frac{1}{2}} S_1 X^{(0,0,m,m-1)}.
      \end{align*} 
By induction, we then have
\begin{align*}
& q^{\frac{1}{2}}  S_{n+1} X^{(0,0,m,m-1)} =q^{\frac{1}{2}}   S_1S_n  X^{(0,0,m,m-1)}  -q^{\frac{1}{2}} X^{(0,0,1,1)} S_{n-1}  X^{(0,0,m,m-1)}   \\
=& S_1 ( x_{n+m+2}  x_{m} -  x_{n+m+1}  x_{m+1} ) \\
&- X^{(0,0,1,1)} ( x_{n+m+1}  x_{m} -  x_{n+m}  x_{m+1} ) \\
=& (x_{n+m+3}   + X^{(0,0,1,1)} x_{n+m+1} ) x_{m}  -    (x_{n+m+2}   + X^{(0,0,1,1)} x_{n+m} ) x_{m+1}    \\
&-  X^{(0,0,1,1)} (x_{n+m+1}  x_{m} -x_{n+m}  x_{m+1})  \\
=& x_{n+m+3}  x_{m} -  x_{n+m+2}  x_{m+1}.
\end{align*} 

(3)\ When $n=0$, it follows from  Theorem \ref{m-2}. When $n=1$, the equality follows from the definition of $S_1$.  According to Lemma \ref{coef} and by induction, we have
\begin{align*}
& S_{n+1}=S_n S_1 - X^{(0,0,1,1)} S_{n-1} \\
=& ( x_0 x_{n+2} - q^{\frac{2n+1}{2}} x_1x_{n+1} y_2)S_1 - ( x_0 x_{n+1} - q^{\frac{2n-1}{2}} x_1x_{n} y_2)X^{(0,0,1,1)} \\
=&x_0 ( x_{n+3} +q x_{n+1} y_1y_2) -q^{\frac{2n+1}{2}}q x_1 x_{n+1} S_1y_2-  ( x_0 x_{n+1} - q^{\frac{2n-1}{2}} x_1x_{n} y_2)qy_1y_2 \\
=& x_0 x_{n+3} +q x_0 x_{n+1} y_1y_2 -q^{\frac{2n+3}{2}}x_1 (x_{n+2} +q x_n y_1y_2 )y_2-(x_0x_{n+1} - q^{\frac{2n-1}{2}} x_1x_{n} y_2)qy_1y_2 \\
=& x_0 x_{n+3} -q^{\frac{2n+3}{2}}x_1 x_{n+2} - q^{\frac{2n+5}{2}}x_1 x_{n}y_1 y_2^2+ q^{\frac{2n+1}{2}}x_1 x_n y_2 y_1y_2\\
=& x_0 x_{n+3} -q^{\frac{2n+3}{2}}x_1 x_{n+2}.
\end{align*}

(4)\  When $m=1$,  it follows from  Theorem \ref{m-2}. 
When $m=2$, we have that
\begin{align*}
& x_{-2} x_1= (S_1 x_{-1}  - X^{(0,0,1,1)}  x_0 )x_1\\
=& S_1 (1 + q^{\frac{1}{2}}  x_0^2 y_1) -  X^{(0,0,1,1)}  x_0 x_1\\
=&  S_1 + q^{\frac{1}{2}} (x_{-1} +  q^{-\frac{1}{2}} X^{(1,0,0,1)} )x_0y_1  - X^{(0,0,1,1)} x_0x_1\\
=& S_1 + q^{\frac{1}{2}}x_{-1}x_0 y_1+x_1y_2x_0y_1-  X^{(0,0,1,1)}  x_0x_1\\
=& S_1+ q^{\frac{1}{2}}x_{-1} x_0y_1.
\end{align*} 

By induction, we have that
\begin{align*}
& x_{-(m+1)} x_1= (  S_1 x_{-m}  - X^{(0,0,1,1)}  x_{-(m-1)} )x_1\\
=& S_1 (S_{m-1} + q^{\frac{1}{2}} x_{-(m-1)} x_0 y_1) -  X^{(0,0,1,1)} ( S_{m-2} + q^{\frac{1}{2}} x_{-(m-2)} x_0 y_1)\\
=&  S_m + X^{(0,0,1,1)} S_{m-2} + q^{\frac{1}{2}} (x_{-m} + X^{(0,0,1,1)} x_{-(m-2)} ) x_0y_1 \\
&- X^{(0,0,1,1)}(S_{m-2} + q^{\frac{1}{2}} x_{-(m-2)}  x_0y_1)\\
=& S_m+ q^{\frac{1}{2}}x_{-m} x_0y_1.
\end{align*} 

(5)\  When $m=1$ and $n\geq 2$, we need to prove
$$x_{-1} x_n=q x_0 x_{n-1} X^{(0,0,1,1)} + S_{n-1}.$$
 When $n=2$, by Lemma \ref{coef}, we have that
\begin{align*}
& x_{-1} x_2= x_{-1} (  x_1 S_1  -  q^{-\frac{1}{2}} x_0y_1 )\\
=&  (q^{\frac{1}{2}} x_0^2 y_1 +1)S_1 -q^{-\frac{1}{2}} x_{-1} x_0 y_1 \\
=& q^{-\frac{1}{2}} x_0^2 S_1 y_1 + S_1 - q^{-\frac{1}{2}} x_{-1} x_0 y_1 \\
=& q^{-\frac{1}{2}} x_0 (x_{-1} + q^{\frac{1}{2}}  X^{(1,0,0,1)} )y_1+ S_1 - q^{-\frac{1}{2}}  x_{-1} x_0 y_1\\
=& x_0  X^{(1,0,0,1)}  X^{(0,0,1,0)} +S_1=q x_0 x_1  X^{(0,0,1,1)}+S_1.
\end{align*} 
By induction, we have that
\begin{align*}
& x_{-1} x_{n+1}= x_{-1} (  x_n S_1  -  x_{n-1}  X^{(0,0,1,1)})\\
=&  (q x_0 x_{n-1}   X^{(0,0,1,1)} +S_{n-1} ) S_1 -(q x_0 x_{n-2}   X^{(0,0,1,1)} +S_{n-2} )  X^{(0,0,1,1)} \\
=& q x_0 (x_{n-1} S_1 )  X^{(0,0,1,1)}  +S_n+S_{n-2}  X^{(0,0,1,1)} - ( q x_0 x_{n-2} X^{(0,0,1,1)} +S_{n-2} ) X^{(0,0,1,1)} \\
=& q x_0 (x_n+x_{n-2} X^{(0,0,1,1)}) X^{(0,0,1,1)} +S_n+S_{n-2}  X^{(0,0,1,1)} \\
& -( q x_0 x_{n-2} X^{(0,0,1,1)} +S_{n-2} ) X^{(0,0,1,1)} \\
=& q x_0 x_n X^{(0,0,1,1)}+S_n.
\end{align*} 

Now we proceed the induction on $m$, we have
\begin{align*}
& x_{-(m+1)} x_n=( S_1 x_{-m} - X^{(0,0,1,1)} x_{-(m-1)} )x_n\\
=& S_1 (q x_{-(m-1)} x_{n-1}  X^{(0,0,1,1)} +S_{m+n-2} ) -  X^{(0,0,1,1)} (q x_{-(m-2)} x_{n-1}  X^{(0,0,1,1)} +S_{m+n-3} )\\
=& q (x_{-m} +X^{(0,0,1,1)} x_{-(m-2)}  ) x_{n-1}  X^{(0,0,1,1)} +S_{m+n-1} + X^{(0,0,1,1)} S_{m+n-3} \\
&- X^{(0,0,1,1)} ( q x_{-(m-2)} x_{n-1}  X^{(0,0,1,1)} +    S_{m+n-3}    )\\
=& q x_{-m} x_{n-1}  X^{(0,0,1,1)} + S_{m+n-1}. 
\end{align*} 
Thus, the proof is completed.
\end{proof}

\begin{remark}
According to   \cite[Definition 4.15, Remark 4.18, Theorem 4.20] {cla}, it follows that the equality in \cite[Lemma 4.16 (a)]{cla} is exactly $x_{n+3} x_{1} = q^{\frac{1}{2}}  S_n y_1 + x_{n+2}x_{2}$ which is a special case in Theorem \ref{calc4} (2) for $m=1$.
\end{remark}

In order to study atomic bases in the rank two cluster algebras of affine types, Sherman and Zelevinsky~\cite{SZ} introduced the $n$-th Chebyshev polynomials of the first kind $F_{n}(x)$. We modify the  definition of $F_{n}(x)$ for $n\in \mathbb{Z}_{\geq 0}$ on $X_\delta$ as follows
  $$
  F_0(X_\delta)=1,F_1(X_\delta)=X_\delta, F_2(X_\delta)=F_1(X_\delta)F_1(X_\delta)-2X^{(0,0,1,1)},$$
  $$F_{n+1}(X_\delta)=F_n(X_\delta)F_1(X_\delta)-X^{(0,0,1,1)} F_{n-1}(X_\delta) \ \ for\ n\geq 2.
 $$
In the following, we denote $F_n(X_\delta)$ by $F_n$ for simplicity. According to  the definition of $F_n$ and the fact that $F_1\in \mathcal{A}_q (Q)$, it is easy to see  $F_n\in \mathcal{A}_q (Q)$ for any $n\geq 1$.

\begin{lemma}\label{first-0}  In $\mathcal{A}_q (Q)$, we have that
\begin{enumerate}
\item $F_m x_0=x_{-m} +  q^{-\frac{2m-1}{2}}y_2 x_m$, for $m\geq 1$;
\item  $F_m x_{-1} =x_{-(m+1)} + q^{-\frac{2m-1}{2}}   X^{(0,0,1,2)} x_{m-1}$, for $m\geq 2$;
\item $F_2 x_{-2} =x_{-4} + X^{(0,0,2,2)}x_0$, and  $F_m x_{-2} =x_{-(m+2)} + q^{-\frac{2m-1}{2}}   X^{(0,0,2,3)} x_{m-2}$, for $m\geq 3$.
\end{enumerate}
\end{lemma}
\begin{proof}
(1)\ The case when $m=1$ follows from Lemma~\ref{induc1} (1).  When $m=2$, we have 
\begin{align*}
&F_2 x_0=X_\delta^2 x_0-2 X^{(0,0,1,1)} x_0\\
=&X_\delta (x_{-1} + q^{-\frac{1}{2}}  y_2x_1) -2 X^{(0,0,1,1)} x_0\\
=& x_{-2} +X^{(0,0,1,1)} x_0 +  q^{-\frac{3}{2}}  y_2  (x_2+ q^{\frac{1}{2}}y_1  x_0) -2 X^{(0,0,1,1)} x_0\\
=& x_{-2}+ q^{-\frac{3}{2}}y_2x_2. 
\end{align*} 
When $m=3$, we have 
\begin{align*}
&F_3 x_0=(X_\delta F_2 -X^{(0,0,1,1)} X_\delta) x_0\\
=&X_\delta (x_{-2} + q^{-\frac{3}{2}}  y_2x_2) -X^{(0,0,1,1)} (x_{-1}  + q^{-\frac{1}{2}}y_2x_1)\\
=& x_{-3} +X^{(0,0,1,1)} x_{-1} +  q^{-\frac{5}{2}}  y_2  (x_3+X^{(0,0,1,1)} x_1) - X^{(0,0,1,1)} (x_{-1}  + q^{-\frac{1}{2}}y_2x_1)\\
=& x_{-3}+ q^{-\frac{5}{2}}y_2x_3. 
\end{align*} 

Suppose that $F_m x_0=x_{-m} +  q^{-\frac{2m-1}{2}}y_2 x_m.$ Then we have that
\begin{align*}
&F_{m+1} x_0=(X_\delta F_m -X^{(0,0,1,1)} F_{m-1}) x_0\\
=&X_\delta (x_{-m} + q^{-\frac{2m-1}{2}}  y_2x_m) -X^{(0,0,1,1)} (x_{-(m-1)}  + q^{-\frac{2m-3}{2}}y_2x_{m-1})\\
=& x_{-(m+1)} +X^{(0,0,1,1)} x_{-(m-1)} +  q^{-\frac{2m+1}{2}}  y_2  (x_{m+1}+X^{(0,0,1,1)} x_{m-1}) \\
&- X^{(0,0,1,1)} (x_{-(m-1)}  + q^{-\frac{2m-3}{2}}y_2x_{m-1})\\
=& x_{-(m+1)}+ q^{-\frac{2m+1}{2}}y_2x_{m+1}. 
\end{align*} 

(2)\ When $m=2$, we have 
\begin{align*}
&F_2 x_{-1}=X_\delta^2 x_{-1}-2 X^{(0,0,1,1)} x_{-1}\\
=&X_\delta (x_{-2} + X^{(0,0,1,1)}x_0) -2 X^{(0,0,1,1)} x_{-1}\\
=& x_{-3} +X^{(0,0,1,1)} x_{-1} +X^{(0,0,1,1)} (x_{-1}+ q^{-\frac{1}{2}}y_2  x_1) -2 X^{(0,0,1,1)} x_{-1}\\
=& x_{-3}+ q^{-\frac{3}{2}} X^{(0,0,1,2)} x_1 . 
\end{align*} 

When $m=3$, we have 
\begin{align*}
&F_3 x_{-1}=X_\delta F_2 x_{-1} -X^{(0,0,1,1)} X_\delta x_{-1}\\
=&X_\delta (x_{-3} + q^{-\frac{3}{2}}  X^{(0,0,1,2)}  x_1 ) -X^{(0,0,1,1)} (x_{-2}  + X^{(0,0,1,1)}  x_0)\\
=& x_{-4} +X^{(0,0,1,1)} x_{-2} +  q^{-\frac{5}{2}}  X^{(0,0,1,2)}  (x_2+ q^{\frac{1}{2}}y_1 x_0) - X^{(0,0,1,1)} (x_{-2}  + X^{(0,0,1,1)}  x_0)\\
=& x_{-4}+ q^{-\frac{5}{2}}  X^{(0,0,1,2)}  x_2.
\end{align*} 
Suppose that $F_m x_{-1} =x_{-(m+1)} + q^{-\frac{2m-1}{2}}   X^{(0,0,1,2)} x_{m-1}$.  Then we have that
\begin{align*}
&F_{m+1} x_{-1}=X_\delta F_m x_{-1} -X^{(0,0,1,1)} F_{m-1} x_{-1}\\
=&X_\delta (x_{-(m+1)} + q^{-\frac{2m-1}{2}}  X^{(0,0,1,2)}  x_{m-1} ) -X^{(0,0,1,1)} ( x_{-m}  +q^{-\frac{2m-3}{2}} X^{(0,0,1,2)}  x_{m-2} )\\
=& x_{-(m+2)} +X^{(0,0,1,1)} x_{-m} +  q^{-\frac{2m+1}{2}}  X^{(0,0,1,2)}  (x_m+ -X^{(0,0,1,1)} x_{m-2}  ) \\
&- X^{(0,0,1,1)} ( x_{-m}  +q^{-\frac{2m-3}{2}} X^{(0,0,1,2)}  x_{m-2} )\\
=& x_{-(m+2)}+ q^{-\frac{2m+1}{2}}  X^{(0,0,1,2)}  x_m.
\end{align*} 

(3)\ The first equality can be calculated as follows
\begin{align*}
&F_2 x_{-2}=X_\delta^2 x_{-2}- 2 X^{(0,0,1,1)} x_{-2}\\
=&X_\delta (x_{-3} + X^{(0,0,1,1)}x_{-1}) -   2 X^{(0,0,1,1)} x_{-2}   \\
=& x_{-4} +X^{(0,0,1,1)} x_{-2} +X^{(0,0,1,1)} (x_{-2}+X^{(0,0,1,1)} x_0 ) -2 X^{(0,0,1,1)} x_{-2}\\
=& x_{-4}+ X^{(0,0,2,2)} x_0. 
\end{align*} 

When $m=3$, we have that 
\begin{align*}
&F_3 x_{-2}=X_\delta F_2 x_{-2}-X^{(0,0,1,1)} X_\delta x_{-2}\\
=&X_\delta (x_{-4} + X^{(0,0,2,2)}x_{0}) - X^{(0,0,1,1)}  (x_{-3} +X^{(0,0,1,1)} x_{-1})        \\
=& x_{-5} +X^{(0,0,1,1)} x_{-3} +X^{(0,0,2,2)} (x_{-1}+   q^{-\frac{1}{2}}   y_2x_1)-  X^{(0,0,1,1)}  (x_{-3} +X^{(0,0,1,1)} x_{-1})  \\
=& x_{-5}+q^{-\frac{5}{2}}X^{(0,0,2,3)} x_1. 
\end{align*} 

When $m=4$, we have that 
\begin{align*}
&F_4 x_{-2}=X_\delta F_3 x_{-2}-X^{(0,0,1,1)} F_2  x_{-2}\\
=&X_\delta (x_{-5} +  q^{-\frac{5}{2}} X^{(0,0,2,3)} x_1) - X^{(0,0,1,1)}  (x_{-4} +X^{(0,0,2,2)} x_{0})        \\
=& x_{-6} +X^{(0,0,1,1)} x_{-4} +q^{-\frac{7}{2}} X^{(0,0,2,3)} (x_{2}+   q^{\frac{1}{2}}   y_1x_0)-   X^{(0,0,1,1)}  (x_{-4} +X^{(0,0,2,2)} x_{0})        \\
=& x_{-6}+q^{-\frac{7}{2}}X^{(0,0,2,3)} x_2. 
\end{align*} 
Suppose that $F_m x_{-2} =x_{-(m+2)} + q^{-\frac{2m-1}{2}}   X^{(0,0,2,3)} x_{m-2}$.
Then we have that
\begin{align*}
&F_{m+1} x_{-2}=X_\delta F_m x_{-2}-X^{(0,0,1,1)} F_{m-1}  x_{-2}\\
=&X_\delta (x_{-(m+2)} +  q^{-\frac{2m-1}{2}} X^{(0,0,2,3)} x_{m-2}) - X^{(0,0,1,1)}  (x_{-(m+1)} +   q^{-\frac{2m-3}{2}} X^{(0,0,2,3)} x_{m-3})        \\
=& x_{-(m+3)} +X^{(0,0,1,1)} x_{-(m+1)} +q^{-\frac{2m+1}{2}} X^{(0,0,2,3)} (x_{m-1}+X^{(0,0,1,1)} x_{m-3}) \\
&- X^{(0,0,1,1)}  (x_{-(m+1)} +   q^{-\frac{2m-3}{2}} X^{(0,0,2,3)} x_{m-3})        \\
=& x_{-(m+3)}+q^{-\frac{2m+1}{2}}X^{(0,0,2,3)} x_{m-1}. 
\end{align*} 
Hence, the proof is finished.
\end{proof}

Similarly, we have the following results.
\begin{lemma}\label{first-1}
In $\mathcal{A}_q (Q)$,  we have that
\begin{enumerate}
\item $x_1 F_m=x_{m+1} + q^{-\frac{2m-1}{2}} x_{-(m-1)} y_1$, for $m \geq 1$;
\item $x_2 F_m=x_{m+2} + q^{-\frac{2m-1}{2}} x_{-(m-2)}  X^{(0,0,2,1)}$, for $m\geq 2$.
\end{enumerate}
\end{lemma}

We can now prove  multiplication formulas between the modified Chebyshev polynomials of the first kind on $X_\delta$ and cluster variables.
\begin{theorem}\label{main} In $\mathcal{A}_q (Q)$,  we have that 
\begin{enumerate}
\item 
\[F_m x_{-n}=\begin{cases}
x_{-(n+m)} + X^{(0,0,m,m)}  x_{-(n-m)} & \text{for\ \  $1\leq m\leq n$;}\\
x_{-(n+m)} + q^{-\frac{2m-1}{2}}X^{(0,0,n,n+1)}  x_{m-n}& \text{for\ \ $n\geq 0,\ n+1\leq m$.}
\end{cases}
\]
\item
\[{\hspace{-1.5cm} } x_n F_m=\begin{cases}
x_{n+m} +  x_{n-m} X^{(0,0,m,m)} & \text{for\ \  $1\leq m< n$;}\\
x_{n+m} + q^{-\frac{2m-1}{2}} x_{-(m-n)} X^{(0,0,n,n-1)}& \text{for\ \ $1\leq n\leq m$.}
\end{cases}
\]
\end{enumerate}  

\end{theorem}
\begin{proof}
(1)\  The cases when $n=0, 1, 2$  follow from Lemmas~\ref{induc1} and~\ref{first-0}.  Assume that $n \geq 2$,  then
we have $X_\delta x_{-(n+1) }= x_{-(n+2)} + X^{(0,0,1,1)} x_{-n}$, and 
\begin{align*}
 & F_2 x_{-(n+1) }= X_\delta^2 x_{-(n+1) }-2 X^{(0,0,1,1)} x_{-(n+1)}      \\ 
= &  X_\delta (x_{-(n+2)} + X^{(0,0,1,1)}  x_{-n}  ) -2 X^{(0,0,1,1)} x_{-(n+1)}      \\ 
=& x_{-(n+3)} + X^{(0,0,1,1)}  x_{-(n+1)}   +X^{(0,0,1,1)} (x_{-(n+1)}  +  X^{(0,0,1,1)} x_{-(n-1)}  )-2 X^{(0,0,1,1)} x_{-(n+1)}   \\
=&  x_{-(n+3)} + X^{(0,0,2,2)}  x_{-(n-1)}.
\end{align*} 
Suppose that $F_{k} x_{-(n+1)} =x_{-(n+1+k)} + X^{(0,0,k,k)} x_{-(n+1-k)}$, for $k\leq n$, then
we have that
\begin{align*}
 &F_{k+1} x_{-(n+1)}= (X_\delta F_{k} - X^{(0,0,1,1)} F_{k-1} ) x_{-(n+1)} \\
 =& X_\delta (x_{-(n+1+k)} + X^{(0,0,k,k)} x_{-(n+1-k)}) -X^{(0,0,1,1)} (x_{-(n+k)} + X^{(0,0,k-1,k-1)} x_{-(n+2-k)})  \\
 =& x_{-(n+2+k)} + X^{(0,0,1,1)}  x_{-(n+k)} + X^{(0,0,k,k)} (x_{-(n+2-k)} + X^{(0,0,1,1)} x_{-(n-k)}) \\
 &- X^{(0,0,1,1)} (x_{-(n+k)} + X^{(0,0,k-1,k-1)} x_{-(n+2-k)}) \\
=& x_{-(n+2+k)} + X^{(0,0,k+1,k+1)} x_{-(n-k)};
     \end{align*} 
Thus, 
\begin{align*}
 &F_{n+2} x_{-(n+1)}= (X_\delta F_{n+1} - X^{(0,0,1,1)} F_{n} ) x_{-(n+1)} \\
 =& X_\delta (x_{-(2n+2)} + X^{(0,0,n+1,n+1)} x_{0}) -X^{(0,0,1,1)} (x_{-(2n+1)} + X^{(0,0,n,n)} x_{-1})  \\
 =& x_{-(2n+3)} + X^{(0,0,1,1)}  x_{-(2n+1)} + X^{(0,0,n+1,n+1)} (x_{-1} + q^{-\frac{1}{2}}  y_2 x_1) \\
 &- X^{(0,0,1,1)} (x_{-(2n+1)} + X^{(0,0,n,n)} x_{-1})  \\
  =& x_{-(2n+3)} +  q^{-\frac{2n+3}{2}}X^{(0,0,n+1,n+2)} x_1;
 \end{align*} 
Similarly, we have that
 $$F_{n+3} x_{-(n+1)}=x_{-(2n+4)} +  q^{-\frac{2n+5}{2}}X^{(0,0,n+1,n+2)} x_2;$$
 $$F_{n+4} x_{-(n+1)} =x_{-(2n+5)} +  q^{-\frac{2n+7}{2}}X^{(0,0,n+1,n+2)} x_3.$$

Now, for $m \geq n+2$,  suppose that 
 $$F_{m} x_{-(n+1)} =x_{-(m+n+1)} +  q^{-\frac{2m-1}{2}}X^{(0,0,n+1,n+2)} x_{m-n-1},$$
 then we have  that
 \begin{align*}
 & F_{m+1} x_{-(n+1)} =(X_\delta F_m- X^{(0,0,1,1)}F_{m-1} ) x_{-(n+1)} \\
 =&   X_\delta (  x_{-(m+n+1)}  + q^{-\frac{2m-1}{2}}  X^{(0,0,n+1,n+2)} x_{m-n-1}  ) \\
 &- X^{(0,0,1,1)} (  x_{-(m+n)}  + q^{-\frac{2m-3}{2}}  X^{(0,0,n+1,n+2)} x_{m-n-2}  ) \\
 =&  x_{-(m+n+2)} +  X^{(0,0,1,1)}   x_{-(m+n)} +q^{-\frac{2m+1}{2}} X^{(0,0,n+1,n+2)} (x_{m-n} + X^{(0,0,1,1)} x_{m-n-2}  )\\
&  -   X^{(0,0,1,1)} ( x_{-(m+n)} +q^{-\frac{2m-3}{2}} X^{(0,0,n+1,n+2)} x_{m-n-2} ) \\
=& x_{-(m+n+2)} +  q^{-\frac{2m+1}{2}}X^{(0,0,n+1,n+2)} x_{m-n}.
\end{align*} 

(2)\ The proof is similar to (1), where we need to use Lemma \ref{first-1} and Theorem \ref{calc2}.
 \end{proof}

\begin{remark}
Since all coefficients appear in multiplication formulas are Laurent polynomials of $q$ which are independent of the choice of the finite field $\mathbb{F}_q$,  generically all multiplication formulas also hold in $\mathcal{A}_{\mathfrak{q}} (Q)$ if we replace $q$ by the indeterminate $\mathfrak{q}$.
\end{remark}

\section{Bar-invariant positive $\mathbb{Z}[ {\mathfrak{q}}^{\pm \frac{1}{2}}] [y_1^{\pm}, y_2^{\pm}]$-bases of $\mathcal{A}_{\mathfrak{q}}(Q)$}
In this section, we use multiplication formulas  established in Section 3 to construct two positive $\mathbb{Z}[ {\mathfrak{q}}^{\pm \frac{1}{2}}] [y_1^{\pm}, y_2^{\pm}]$-bases. 

\begin{definition}
An element in $\mathcal{A}_{\mathfrak{q}} (Q)$ is called positive if the coefficients of its Laurent expansion associated to any cluster belong to
$\mathbb{Z}_{\geq 0}[ {\mathfrak{q}}^{\pm \frac{1}{2}}] [y_1^{\pm}, y_2^{\pm} ]$.
\end{definition}

\begin{definition}
A basis of $\mathcal{A}_{\mathfrak{q}} (Q)$ is called a positive $\mathbb{Z}[ {\mathfrak{q}}^{\pm \frac{1}{2}}] [y_1^{\pm}, y_2^{\pm}]$-basis if its
structure constants belong to $\mathbb{Z}_{\geq 0}[ {\mathfrak{q}}^{\pm \frac{1}{2}}] [y_1^{\pm}, y_2^{\pm} ]$.
\end{definition}

Denote the following 
$$\mathcal{S}=\{ \text{cluster monomials} \} \sqcup \{ S_n | n\geq 1\}. $$

\begin{theorem}\label{positive}
The set $\mathcal{S}$ is a bar-invariant positive  $\mathbb{Z}[ {\mathfrak{q}}^{\pm \frac{1}{2}}] [y_1^{\pm}, y_2^{\pm}]$-basis of $\mathcal{A}_{\mathfrak{q}} (Q)$.
\end{theorem}

\begin{proof}
According to Proposition~\ref{calc1}, Theorem~\ref{calc2},  Theorem~\ref{calc4}, Lemma~\ref{coef} and Lemma~\ref{bar-s}, we can deduce
that the $\mathbb{Z}[ {\mathfrak{q}}^{\pm \frac{1}{2}}] [y_1^{\pm}, y_2^{\pm}]$-combinations of the elements in $\mathcal{S}$ is $\mathcal{A}_{\mathfrak{q}} (Q)$. 

In order to prove the $\mathbb{Z}[ {\mathfrak{q}}^{\pm \frac{1}{2}}] [y_1^{\pm}, y_2^{\pm}]$-independency of these elements, we  define a partial order $\leq$ on $\mathbb{Z}^2$ as follows: $(r_1,r_2) \leq (s_1,s_2)$ if $r_i \leq s_i$ for $1\leq i \leq 2$. Moreover if  $r_i < s_i$ for some $i$, we will write $(r_1,r_2)< (s_1,s_2)$. According to  Proposition \ref{calc1},  every element $s_n$ has a minimal non-zero term $a_nX^{-\underline{dim}R(n)}$ where $a_n\in \mathbb{Z}[ {\mathfrak{q}}^{\pm \frac{1}{2}}] [y_1^{\pm}, y_2^{\pm}]$. According to  Theorem \ref{calc2},   every element $x_m$ has a minimal non-zero term $b_mX^{-\underline{dim}V(m)}$ where $b_m\in\mathbb{Z}[ {\mathfrak{q}}^{\pm \frac{1}{2}}] [y_1^{\pm}, y_2^{\pm}]$.  It follows from \cite[Proposition 3.1]{SZ}, there exists a bijection between the set of all minimal non-zero terms in the elements in $\mathcal{S}$ and  $\mathbb{Z}^2$, which implies that the elements in  $\mathcal{S}$ are $\mathbb{Z}[ {\mathfrak{q}}^{\pm \frac{1}{2}}] [y_1^{\pm}, y_2^{\pm}]$-independent.

By Lemma~\ref{bar-s} and the fact that cluster monomials are bar-invariant, we know that any element in $\mathcal{S}$ is bar-invariant. Again from Proposition~\ref{calc1}, Theorem~\ref{calc2},  Theorem~\ref{calc4}, Lemma~\ref{coef} and Lemma~\ref{bar-s},   we can deduce that  the
structure constants belong to $\mathbb{Z}_{\geq 0}[ {\mathfrak{q}}^{\pm \frac{1}{2}}] [y_1^{\pm}, y_2^{\pm} ]$.
\end{proof}

\begin{remark}\label{positive-element}
Using the same arguments as \cite[Corollary 3.3.10]{KQ} \cite[Corollary 8.3.3]{fanqin1}, it follows that all elements in
$\mathcal{S}$ are positive.
\end{remark}

\begin{definition}
A basis of $\mathcal{A}_{\mathfrak{q}}(Q)$ is called the atomic $\mathbb{Z}[ {\mathfrak{q}}^{\pm \frac{1}{2}}] [y_1^{\pm}, y_2^{\pm}]$-basis if  positive elements in $\mathcal{A}_{\mathfrak{q}}(Q)$ are exactly $\mathbb{Z}_{\geq 0}[ {\mathfrak{q}}^{\pm \frac{1}{2}}] [y_1^{\pm}, y_2^{\pm} ]$-combinations of these basis elements.
\end{definition}
It follows from the definition that the atomic basis, if exists, must be a positive basis.

Denote the following 
$$\mathcal{B}=\{ \text{cluster monomials} \} \sqcup \{ F_n | n\geq 1\}. $$

\begin{theorem}\label{atomic}
The set $\mathcal{B}$ is the bar-invariant atomic  $\mathbb{Z}[ {\mathfrak{q}}^{\pm \frac{1}{2}}] [y_1^{\pm}, y_2^{\pm}]$-basis  of $\mathcal{A}_{\mathfrak{q}} (Q)$.
\end{theorem}
\begin{proof}
By Theorem~\ref{positive} and the relation between $S_n$ and $F_n$, we have that $\mathcal{B}$ is a $\mathbb{Z}[ {\mathfrak{q}}^{\pm \frac{1}{2}}] [y_1^{\pm}, y_2^{\pm}]$-basis. By applying similar discussions as in Theorem~\ref{positive},we can show that the elements in $\mathcal{B}$ are bar-invariant. Note that cluster variables are positive (see also Remark \ref{positive-element}), then
by Theorem~\ref{main}, we have that $F_n$ is a positive element for  $n\geq 1$. Thus,  
all elements in $\mathcal{B}$ are positive.

Let $Y$ be a positive element in $\mathcal{A}_{\mathfrak{q}}(Q)$, then $Y=\sum_{b\in \mathcal{B}} \lambda_b b$\,
for\, some\, $ \lambda_b \in \mathbb{Z}[ {\mathfrak{q}}^{\pm \frac{1}{2}}] [y_1^{\pm}, y_2^{\pm}]$.  We need to prove $\lambda_b \in \mathbb{Z}_{\geq 0}[\displaystyle {\mathfrak{q}}^{\pm \frac{1}{2}}][y_1^{\pm}, y_2^{\pm}]$.
Firstly, we consider the case when the element $b$ is a cluster monomial. Without loss of generality,  suppose that $b$ is a cluster monomial in  the cluster $\{ x_m, x_{m+1}\}$. If  $b$ appears in the $\mathbb{Z}_{\geq 0}[\displaystyle {\mathfrak{q}}^{\pm \frac{1}{2}}][y_1^{\pm}, y_2^{\pm}]$-expansion of some other basis element $b'$ of $\mathcal{B}$ associated to the cluster $\{ x_m, x_{m+1}\}$, then $b|_{\mathfrak{q}=1,y_1=1,y_2=1}$ appears in the $\mathbb{Z}_{\geq 0}$-expansion of some other basis element $b'|_{\mathfrak{q}=1,y_1=1,y_2=1}$ associated to the cluster $\{ x_m, x_{m+1}\}|_{\mathfrak{q}=1,y_1=1,y_2=1}$ due to the positivity of the elements in $\mathcal{B}$. However this can't happen in the corresponding classical cluster algebra, for more details, refer to \cite[Section 4, Proof of
(B3) in Theorem 1.2]{DTh} or ~\cite[Section 5, Proof of
(5.10)]{SZ}. Thus by the assumption that $Y$ is  positive  associated to the cluster $\{ x_m, x_{m+1}\}$, we have  $\lambda_b \in \mathbb{Z}_{\geq 0}[\displaystyle {\mathfrak{q}}^{\pm \frac{1}{2}}][y_1^{\pm}, y_2^{\pm}]$.
Now,  we consider the case when $b=F_l$ for some $l \geq 1$.  
By using the same discussions and references as above, we can find a Laurent monomial $Y_{l}$ in certain  cluster expansion of  $F_{l}$ with coefficients  ${\mathfrak{q}}^{\frac{l_0}{2}}y^{l_1}_1y^{l_2}_2$ for some $l_0 , l_1, l_2 \in \mathbb{Z}$, but $Y_{l}$ doesn't appear in this cluster expansion of any other basis element in the above sum terms. Thus we have $\lambda_{b} {\mathfrak{q}}^{\frac{l_0}{2}}y^{l_1}_1y^{l_2}_2\in \mathbb{Z}_{\geq 0}[\displaystyle {\mathfrak{q}}^{\pm \frac{1}{2}}][y_1^{\pm}, y_2^{\pm}]$ by  using the assumption that $Y$ is  positive associated to this cluster. It follows that  $\lambda_b \in \mathbb{Z}_{\geq 0}[\displaystyle {\mathfrak{q}}^{\pm \frac{1}{2}}][y_1^{\pm}, y_2^{\pm}]$.
The proof is completed.
\end{proof}

\begin{remark}
By specializing $\mathfrak{q}=1$ in Theorems~\ref{positive} and~\ref{atomic}, we obtain a positive integral basis and the atomic basis of the classical cluster algebra associated to the Kronecker quiver with principal coefficients.  Note that  Dupont constructed these bases  \cite[Theorem 7.3, Theorem 7.4]{Dup0} in the classical Kronecker cluster algebra with opposite principal coefficients.
\end{remark}

\section*{Acknowledgments}
Ming Ding was supported by NSF of China (No. 11771217) and Fan Xu was supported by NSF of China (No. 12031007). We thank the referees for their very helpful comments and suggestions.

\end{document}